\documentclass[reqno,11pt]{amsart}
\usepackage[T1]{fontenc}
\usepackage{lmodern}
\usepackage{amsmath,amssymb,amsfonts,amsthm,amscd,amsxtra,mathrsfs}
\usepackage{indentfirst,mathtools,thmtools}
\usepackage{tikz,tikz-3dplot,subcaption,xcolor}
\usetikzlibrary{arrows.meta}
\usepackage[top=3.5cm,bottom=4.0cm,right=2.5cm,left=2.5cm]{geometry}
\usepackage[title]{appendix}
\usepackage[hyphens]{url}
\usepackage{hyperref}
\usepackage{bookmark}
\allowdisplaybreaks
\mathtoolsset{showonlyrefs=true}
\numberwithin{equation}{section}
\title{\textbf{Multi-marginal optimal transport on sub-Riemannian manifolds}}

\author{Haowei Lin}
\address{Graduate School of Science, the University of Osaka, Osaka 560-0043, Japan}
\email{u466248a@ecs.osaka-u.ac.jp}
\date{\today}
\theoremstyle{plain}
\newtheorem{theorem}{Theorem}[section]

\newtheorem{proposition}[theorem]{Proposition}
\newtheorem{corollary}[theorem]{Corollary}

\theoremstyle{definition}
\newtheorem{definition}[theorem]{Definition}
\newtheorem{example}[theorem]{Example}

\newtheorem{remark}[theorem]{Remark}

\newcommand \supp {\operatorname{supp}}

\newcommand \m {\mathfrak{m}}

\newcommand \Cut {\mathrm{Cut}}

\newcommand{\R}{\mathbb{R}}
\newcommand{\Pc}{\mathcal{P}_{c}}
\newcommand{\Pac}{\mathcal{P}^{\mathrm{ac}}_{c}}
\newcommand{\Ann}{\operatorname{Ann}}

\begin{document}
\bibliographystyle{plain}
\nocite{*}

\begin{abstract}
In this paper, we study multi-marginal optimal transport on ideal sub-Riemannian manifolds. We establish a sub-Riemannian barycentric balance condition and prove the existence of a solution to the Monge problem under an additional assumption. We also investigate the absolute continuity of Wasserstein barycenters. Our results extend and improve upon those of Brendan Pass, Andrea Pinamonti, and Mattia Vedovato \cite{PPV21} on the Heisenberg group, answering two open questions raised in their work.
\end{abstract}
\maketitle
\thispagestyle{empty}
\pagenumbering{Roman}
\setcounter{page}{1}
\tableofcontents
\setcounter{page}{1}
\pagenumbering{arabic}
\tdplotsetmaincoords{70}{120}

\section{Introduction}\label{sec:introduction}

Let $(X,d,\mathfrak{m})$ be a metric measure space, where $\mathfrak{m}$ is a Borel measure on $X$. We denote by $\mathcal{P}(X)$ the space of Borel probability measures on $X$, and by $\mathcal{P}_c(X)$ the subset consisting of compactly supported probability measures. Moreover, $\mathcal{P}^{\mathrm{ac}}_c(X)$ denotes the set of compactly supported probability measures that are absolutely continuous with respect to $\mathfrak{m}$.

Given $m$ probability measures $\mu_1,\mu_2,\ldots,\mu_m$ on $X$, consider the following multi-marginal optimal transport problem (MMOT): find a probability measure $\gamma\in\mathcal{P}(X^m)$ satisfying
\begin{equation}
    (\pi_i)_\sharp\gamma=\mu_i,
    \qquad i=1,\ldots,m,
\end{equation}
which minimizes
\begin{equation}\label{kp}\tag{KP}
    \int_{X^m}c(x_1,\ldots,x_m)
    \,d\gamma(x_1,\ldots,x_m),
\end{equation}
where $c:X^m\to\mathbb{R}$ is a given cost function. In this paper, we are mainly interested in the barycentric cost
\begin{equation}\label{eq:barycentric-cost}
    c(x_1,\ldots,x_m)
    :=
    \inf_{y\in X}\sum_{i=1}^m d^2(x_i,y).
\end{equation}
Any point $y\in X$ attaining this minimum is called a \emph{barycenter} of $x_1,\ldots,x_m$.

Under mild assumptions on the cost function, the Kantorovich problem \eqref{kp} admits an optimal plan; see \cite{KP15,Villani09}. A further question is whether an optimal plan is supported on a graph. This leads to the Monge problem, which seeks a measurable $(m-1)$-tuple of maps $F=(F_2,\ldots,F_m)$, with $(F_i)_\sharp\mu_1=\mu_i$ for $i=2,\ldots,m$, minimizing
\begin{equation}\label{mp}\tag{MP}
  \int_X c(x_1,F_2(x_1),\ldots,F_m(x_1))\,d\mu_1(x_1).
\end{equation}

The Monge problem \eqref{mp} has been studied in a variety of geometric settings. Existence and uniqueness were first established in Euclidean space \cite{GS98} and subsequently extended to Riemannian manifolds \cite{KP15}. In the nonsmooth setting, building on Ohta's work on barycenters in Alexandrov spaces \cite{Ohta12}, Jiang \cite{Jia17} established existence and uniqueness for the corresponding Monge problem. From a different perspective, Han, Liu and Zhu \cite{HLZ24} studied MMOT in RCD spaces.

Extending this theory to sub-Riemannian manifolds presents additional difficulties. Pass, Pinamonti, and Vedovato \cite{PPV21} established existence and uniqueness for the Monge problem on the Heisenberg group under two additional technical assumptions. The first requires that each marginal $\mu_i$ assigns mass $0$ to the set of barycenters of points in the supports of the marginals, which ensures that the map sending optimal tuples to barycenters is injective almost everywhere. This injectivity holds in the Riemannian setting but is no longer automatic on the Heisenberg group, where the squared distance fails to be semiconcave on the diagonal. The second requires all marginals to be absolutely continuous with respect to Lebesgue measure, whereas the classical Euclidean and Riemannian results require absolute continuity of only one marginal.

In this paper, we study barycentric MMOT on ideal sub-Riemannian manifolds, where semiconcavity of the squared distance on the diagonal is likewise unavailable in general. To investigate how this affects the correspondence between barycenters and optimal tuples, we use the Lagrange multiplier rule and other tools from control theory to characterize barycenters. Using this characterization, we obtain a solution to \eqref{mp} under the same assumption (Condition A; see Definition~\ref{condA}) as in \cite[Corollary~29]{PPV21}, without assuming that all marginals are absolutely continuous with respect to $\m$. We also study the absolute continuity of Wasserstein barycenters in the spirit of Kim and Pass \cite{KP17}.
We first prove this when all marginals are absolutely continuous.
We then extend the result to the case where only one marginal
(for convenience, we take the first one) is absolutely continuous.
For the remaining marginals, we impose Condition B
(Definition~\ref{condAA}), which requires their singular mass
at static points to be concentrated on an at most countable set. This condition is substantially weaker than both Condition~A
and the requirement that all marginals be absolutely continuous.
Based on these results, we obtain a second existence result for \eqref{mp}.

This article is organized as follows. Sections~\ref{sec:sr} and~\ref{sec:mmot} review the necessary background on sub-Riemannian geometry and multi-marginal optimal transport. Section~\ref{sec:structure} establishes a characterization of barycenters (Theorem~\ref{thm:barycenter-rigidity}) and gives a sufficient condition for the existence of a Monge solution (Theorem~\ref{mppp}). We also present counterexamples to injectivity of the barycenter map. Section~\ref{sec:ac} establishes the absolute continuity of Wasserstein barycenters (Theorem~\ref{thm:ac-barycenter-draft} and Theorem~\ref{54}) and uses this result to prove existence and uniqueness for the Monge problem (Theorem~\ref{thm:monge-final}). We end this article by introducing some further questions in Section~\ref{further-discussion}.

\section{Sub-Riemannian geometry}\label{sec:sr}

Let $M$ be a connected $n$-dimensional smooth manifold, and let
$\mathcal D\subset TM$ be a smooth distribution generated by
global smooth vector fields $X_1,\ldots,X_k$.
We equip $\mathcal D$ with the metric $g$ defined by
\begin{equation}
    g_x(v,v)
    =
    \min\left\{
      \sum_{i=1}^k u_i^2:
      u\in\mathbb R^k,\quad
      v=\sum_{i=1}^k u_iX_i(x)
    \right\},
    \quad v\in\mathcal D_x.
\end{equation}
We additionally assume that $\mathcal D$ is bracket-generating; namely, at each
$x\in M$, the vectors $X_1(x),\ldots,X_k(x)$ and their iterated
Lie brackets evaluated at $x$ span $T_xM$.
The step at $x$ is the smallest bracket length needed to span $T_xM$.
We also assume that $\mathcal D$ has step $s$ at every point.
Then $(M,\mathcal D,g)$ is called a
\emph{sub-Riemannian manifold}. We refer to \cite{ABB20} for the standard background, and to \cite{Rif14}
for sub-Riemannian optimal transport theory.

An \emph{admissible curve} $\gamma\colon[0,1]\to M$ is an
absolutely continuous path for which there exists a control
$u=(u^1,\ldots,u^k)\in L^2([0,1];\R^k)$ satisfying
\[
  \dot\gamma(t)=\sum_{i=1}^k u^i(t)X_i(\gamma(t))
  \qquad\text{for a.e. }t\in[0,1].
\]
Its \emph{length} is given by
\[
  L(\gamma)=\int_0^1 g(\dot\gamma(t),\dot\gamma(t))^{1/2}\,dt.
\]
The \emph{sub-Riemannian distance} is then defined by
\[
  d(x,y)
  :=\inf\bigl\{L(\gamma):\gamma(0)=x,\ \gamma(1)=y,
  \ \gamma\text{ is admissible}\bigr\}.
\]
The following theorem ensures that this distance is well defined.
\begin{proposition}[Chow--Rashevskii; {\cite{ABB20}}]
Let $(M,\mathcal D,g)$ be a sub-Riemannian manifold. Then
$d\colon M\times M\to\R$ is finite and continuous.
\end{proposition}

A \emph{geodesic} is an admissible curve that locally minimizes
the length between its endpoints.
We call a curve attaining the minimum in the definition of
$d(x,y)$ a \emph{minimizing geodesic}. Moreover, it is well known that a constant-speed admissible curve $\gamma$ on $[0,1]$ minimizes length between its endpoints if and only if it minimizes the energy
\[
  J(\gamma)=\frac12\int_0^1 g(\dot\gamma(t),\dot\gamma(t))\,dt.
\]
A control $u\in L^2([0,1];\mathbb{R}^k)$ associated with the above curve $\gamma$ is called minimal if $\|u\|_{L^2}^2=2J(\gamma)$. Such a control exists and is unique; see the discussion following \cite[Definition~3.8]{ABB20}.
Throughout this article, all geodesics are parametrized
at constant speed, and, unless otherwise stated, their associated
controls are taken to be minimal. In this setting, the energy
can also be regarded as a functional of the control, and we write
$J(u):=\frac12\|u\|_{L^2}^2$.
\subsection{Geodesics}
We introduce two characterizations of minimizing geodesics. The first is given by the Pontryagin maximum principle. 

For $i=1,\ldots,k$, set $h_i(\lambda):=\langle\lambda,X_i\rangle$, where $\lambda\in T^*M$.
The sub-Riemannian Hamiltonian is
\begin{equation}
\begin{gathered}
H\colon T^*M\to\mathbb{R},\\
H(\lambda)
:=\frac12\sum_{i=1}^k h_i(\lambda)^2
=\frac12\sum_{i=1}^k\langle\lambda,X_i\rangle^2,
\qquad \lambda\in T^*M.
\end{gathered}
\end{equation}
For a smooth function $f$ on $T^*M$, the Hamiltonian vector field $\vec f$ is defined by
$\omega(\cdot\,,\vec f)=d f,$
where $\omega$ is the canonical symplectic form on $T^*M$. The Pontryagin maximum principle gives a necessary condition for geodesics.

\begin{theorem}[Pontryagin maximum principle; see {\cite[Theorem~4.20]{ABB20}}]
Let $(M,\mathcal D,g)$ be a sub-Riemannian manifold and let $\gamma_u$ be a
minimizing geodesic with control $u$. Then there exists a Lipschitz extremal lift
$\lambda\colon[0,1]\to T^*M$ such that $\pi(\lambda(t))=\gamma_u(t)$ and
\[
  \dot\lambda(t)=\sum_{i=1}^k u^i(t)\vec h_i(\lambda(t)),\qquad\text{for a.e. }t\in[0,1].
\]
Moreover, one of the following conditions holds:
\begin{itemize}
  \item[(1)] $u^i(t)=h_i(\lambda(t))$ for $i=1,\ldots,k$ and a.e.\ $t$;
consequently, $\dot\lambda(t)=\vec H(\lambda(t))$.
  \item[(2)] $\langle \lambda(t),X_i(\gamma_u(t))\rangle=0$ for $i=1,\ldots,k$ and all $t\in[0,1]$; in this case $\lambda\neq 0$.
\end{itemize}
\end{theorem}

Lifts $\lambda(t)$ satisfying \textup{(1)} are called \emph{normal extremals}, and lifts satisfying \textup{(2)} are called \emph{abnormal extremals}. Their projections are called \emph{normal geodesics} and \emph{abnormal geodesics}, respectively. A normal geodesic admitting no abnormal lift is called \emph{strictly normal}.
The theorem above enables us to define the exponential map on a complete sub-Riemannian manifold, which describes normal geodesics:
\begin{equation}
  \exp_x(\lambda)=
  \begin{cases}
    \pi\bigl(e^{\vec H}(\lambda)\bigr),
      & \lambda\in T_x^*M\setminus\Ann(\mathcal D_x),\\
    x,
      & \lambda\in\Ann(\mathcal D_x),
  \end{cases}
\end{equation}
where $e^{t\vec H}$ denotes the Hamiltonian flow, and the annihilator of the distribution $\mathcal D$ is defined by
\begin{equation}
  \Ann(\mathcal D)
  :=\{p\in T^*M:p(v)=0
      \text{ for every }v\in\mathcal D_{\pi(p)}\}.
\end{equation}

The second characterization uses the Lagrange multiplier rule.
\begin{definition}[Endpoint map]
  Let $(M,\mathcal{D},g)$ be a sub-Riemannian manifold and fix a point $x\in M$.
  Let $\mathcal{U}_x\subset L^2([0,1];\mathbb{R}^k)$ be the open set of all control functions for which the corresponding curve $\gamma_u(t)$ starting at $x$ is defined on $[0,1]$. We define the \emph{endpoint map} based at $x$ by
  \begin{equation}
    E_x:\mathcal{U}_x\to M,\qquad E_x(u)=\gamma_u(1).
  \end{equation}
\end{definition}

\begin{theorem}[Lagrange multiplier rule; {\cite[Theorem~B.2]{Rif14}}]
Let $(M,\mathcal D,g)$ be a sub-Riemannian manifold, and let $\gamma_u$ be a
minimizing geodesic from $x$ to $y$. Then there exists
$(\eta,\lambda_0)\ne0$, with
$\eta\in T_y^*M$ and $\lambda_0\in\{1,0\}$, such that
\[
  \eta^*D_uE_x=\lambda_0D_uJ.
\]
When $\lambda_0=1$, the geodesic is normal, with $x=\exp_y(-\eta)$; when $\lambda_0=0$, it is
abnormal. 
\end{theorem}

We conclude this subsection with a result that recovers a minimizing geodesic from a differentiable lower support of the squared distance.
\begin{proposition}[{\cite[Proposition~5.5]{FR10}}]\label{prop:lower-support}
  Let $(M,\mathcal D,g)$ be a complete sub-Riemannian manifold and $x\in M$. Assume there is a function $\phi:M\to \mathbb{R}$ differentiable at $y$ such that
  \begin{equation}
    \phi(y)=d^2(x,y)\quad \text{and} \quad d^2(x,z)\ge \phi(z)\quad \forall z\in M.
  \end{equation}
  Then there exists a unique minimizing geodesic $\gamma(t)=\exp_y(-\frac{t}{2}d_y\phi)$ between $x=\gamma(1)$ and $y=\gamma(0)$.
\end{proposition}

\subsection{Regularity of the sub-Riemannian distance}
The sub-Riemannian distance is continuous, but it need not be locally Lipschitz in smooth coordinates. The Ball--Box theorem describes its local behavior in suitable coordinates.
\begin{theorem}[Ball--Box theorem; {\cite[Chapter~10]{ABB20}}]\label{thm:ball-box}
Let $(M,\mathcal D,g)$ be a sub-Riemannian manifold. For every $p\in M$, there exist a local coordinate chart $(U,z)$ centered at $p$ and a constant $C>0$ such that
\begin{equation}
  \frac{1}{C}\sum_{i=1}^n |z^i|^{1/w_i}\le d(p,q) \le C\sum_{i=1}^n |z^i|^{1/w_i},
\end{equation}
where $q=(z^1,\ldots,z^n)$ and $p=(0,\ldots,0)$, and $w_i\ge 1$, $i=1,\ldots,n$, are constants determined by the distribution at $p$.
\end{theorem}
\begin{remark}
For brevity, we have stated the Ball--Box theorem in a simplified form.
The weights $w_i$ are determined precisely by the growth vector of the
distribution at the reference point. We refer to
\cite[Theorem~10.67]{ABB20}
for the definitions and the precise statement.
\end{remark}
The Ball--Box theorem shows that the distance function cannot be locally
Lipschitz in smooth coordinates near $(x,x)$ unless $\mathcal D_x=T_xM$.
It also yields the following local H\"older estimate as a direct corollary.
\begin{proposition}[{\cite[Proposition~11.1]{ABB20}}]
  Let $(M,\mathcal D,g)$ be a sub-Riemannian manifold. Fix $x\in M$. In a sufficiently small coordinate neighborhood $U$ of $x$, there is a constant $C>0$ such that, for $p,q\in U$,
  \begin{equation}
    d(p,q)\le C|p-q|^{1/s},
  \end{equation}
where $s=s(x)$ represents the step of the distribution at $x$.
\end{proposition}

\begin{definition}[Smooth point and cut locus]
Let $(M,\mathcal D,g)$ be a complete sub-Riemannian manifold. We say that $y\in M$ is a \emph{smooth point} with respect to $x$
if there exists a unique minimizing geodesic from
$x$ to $y$, this geodesic is strictly normal, and $x$ and $y$
are not conjugate along it. The \emph{cut locus} $\Cut(x)$ of $x$ is the complement of the set of
smooth points of $x$. Similarly,
\[
  \Cut(M):=\bigl\{(x,y)\in M\times M:y\in\Cut(x)\bigr\}.
\]
\end{definition}

The sub-Riemannian distance behaves well outside the cut locus.

\begin{proposition}[see {\cite[Chapter~11]{ABB20}}]
The set of smooth points is open and dense in $M$, and $d^2$ is
smooth on $M\times M\setminus\Cut(M)$.
\end{proposition}

The preceding proposition suggests that the distance enjoys better regularity
in the absence of abnormal minimizing geodesics. This motivates the following
ideality assumption, which excludes nonconstant abnormal minimizing geodesics.
\begin{definition}
We say that $(M,\mathcal D,g)$ is \emph{ideal} if it is complete and admits no nontrivial
abnormal minimizing geodesics.
\end{definition}
\begin{proposition}[{see \cite[Theorem~5.9]{FR10}}]
Let $(M,\mathcal D,g)$ be an ideal sub-Riemannian manifold. Then
$d^2$ is locally semiconcave outside the diagonal. In
particular, $d^2$ is locally Lipschitz outside the diagonal.
\end{proposition}

On the other hand, under the ideality assumption, the cut locus can be characterized as the set of points where the squared distance function fails to be semiconvex, as in Riemannian geometry.

\begin{theorem}[see {\cite[Theorem~6]{BR19}}]
Let $(M,\mathcal D,g)$ be ideal and let $y\ne x$. Then $y\in\Cut(x)$ if and only if
$d^2(x,\cdot)$ fails to be semiconvex at $y$.
\end{theorem}

\section{The multi-marginal Kantorovich problem}\label{sec:mmot}

Let $(M,\mathcal D,g)$ be a sub-Riemannian manifold equipped with a fixed
smooth positive measure $\m$, and let $\mu_i\in\Pc(M)$ for $i=1,\ldots,m$.
We consider the Kantorovich problem \eqref{kp} introduced in
Section~\ref{sec:introduction}.
The dual of \eqref{kp} is to find
$\varphi=(\varphi_1,\ldots,\varphi_m)$ maximizing
\begin{equation}\label{dp}\tag{DP} 
    \sum_{i=1}^m\int_M\varphi_i(x_i)\,d\mu_i(x_i)
\end{equation}
subject to
\[
  \sum_{i=1}^m\varphi_i(x_i)\le c(x_1,\ldots,x_m),
\]
where $c:M^m\to \mathbb{R}$ is the barycentric cost defined in \eqref{eq:barycentric-cost}.
\begin{definition}[$c$-conjugate tuple]
We say that $\varphi$ is \emph{$c$-conjugate} if, for every $i=1,\ldots,m$,
\[
  \varphi_i(x_i)
  =\inf_{x_j\in M,\,j\ne i}
   \left\{c(x_1,\ldots,x_m)-\sum_{j\ne i}\varphi_j(x_j)\right\}.
\]
\end{definition}
For convenience, we recall the associated notion of cost concavity.
\begin{definition}
  A function $\phi:M\to \mathbb{R}$ is \emph{$d^2$-concave} if there exists a function $\psi:M\to \mathbb{R}\cup\{-\infty\}$ with $\psi\not\equiv -\infty$ 
  such that for all $x \in M$,
  \begin{equation}
    \phi(x)=\inf_{y\in M}\{d^2(x,y)-\psi(y)\}.
  \end{equation}
  If $\phi$ is $d^2$-concave, we define the \emph{$d^2$-superdifferential} of $\phi$ at $x$ by
  \begin{equation}
    \partial^{d^2}\phi(x):=\{y\in M:\phi(x)+\psi(y)=d^2(x,y)\}.
  \end{equation}
  Moreover, the $d^2$-superdifferential of $\phi$ is given by
  \begin{equation}
    \partial^{d^2}\phi=\{(x,y)\in M\times M:y\in\partial^{d^2}\phi(x)\}.
  \end{equation}
\end{definition}

The existence of solutions to the Kantorovich problem and its dual is
well known under mild assumptions; we refer the reader to
\cite[Theorem~2.1]{KP15} for the Riemannian case.

\begin{theorem}
Let $(M,\mathcal D,g)$ be a sub-Riemannian manifold, and let
$\mu_1,\ldots,\mu_m\in\mathcal P_c(M)$.
There exist a $c$-conjugate tuple
$\varphi=(\varphi_1,\ldots,\varphi_m)$ solving \eqref{dp}, and an optimal plan $\gamma\in\mathcal P_c(M^m)$ solving \eqref{kp},
with
\[
  \sum_i\varphi_i(x_i)=c(x_1,\ldots,x_m),
  \qquad\text{for }\gamma\text{-a.e. }x.
\]
Set
\[
  \Gamma_\varphi
  :=\left\{(x_1,\ldots,x_m)\in M^m:
  \sum_{i=1}^m\varphi_i(x_i)=c(x_1,\ldots,x_m)\right\}.
\]
Then $\Gamma_\varphi$ is $c$-cyclically monotone and
$\supp\gamma\subset\Gamma_\varphi$. 
Here, $c$-cyclic monotonicity means that, for every finite family
$x^1,\ldots,x^N\in\Gamma_\varphi$ and every choice of permutations
$\sigma_2,\ldots,\sigma_m$ of $\{1,\ldots,N\}$,
\[
  \sum_{k=1}^N c(x_1^k,\ldots,x_m^k)
  \le
  \sum_{k=1}^N
  c\bigl(x_1^k,x_2^{\sigma_2(k)},\ldots,x_m^{\sigma_m(k)}\bigr).
\]
\end{theorem}

The first property relates the MMOT potentials to the squared-distance cost.

\begin{proposition}\label{prop:potential-barycenter}
Let $\varphi$ be a $c$-conjugate solution to \textup{(DP)}. Then each $\varphi_i$ is
$d^2$-concave, and every barycenter $y$ of
$(x_1,\ldots,x_m)\in\supp\gamma$ belongs to
$\partial^{d^2}\varphi_i(x_i)$.
\end{proposition}

\begin{proof}
Since $\varphi$ is $c$-conjugate, for each $i$,
\begin{align*}
  \varphi_i(x_i)
  &=\inf_{x_j,\,j\ne i}
    \left\{c(x_1,\ldots,x_m)-\sum_{j\ne i}\varphi_j(x_j)\right\}\\
  &=\inf_y\left\{
    d^2(x_i,y)+\sum_{j\ne i}
    \inf_{x_j}\bigl(d^2(x_j,y)-\varphi_j(x_j)\bigr)
    \right\}.
\end{align*}
Set
\[
  \psi_i(y)
  :=\sum_{j\ne i}
  \sup_{x_j}\bigl\{\varphi_j(x_j)-d^2(x_j,y)\bigr\}.
\]
The conclusion follows immediately from this representation.
\end{proof}

Provided that the squared sub-Riemannian distance is locally Lipschitz off the diagonal and $\mu_1\in\mathcal{P}^{\mathrm{ac}}_c(M)$, the preceding proposition, together with sub-Riemannian optimal transport theory,
implies that, for $\mu_1$-a.e. $x_1\in\supp\mu_1$, the barycenter $y$
is uniquely determined by $x_1$ and $\varphi_1$. More precisely, define
\begin{equation}\label{eq:barycenter-map}
  T(x_1):=
  \begin{cases}
    \exp_{x_1}\bigl(-\frac12 d_{x_1}\varphi_1\bigr),
      & x_1\notin\partial^{d^2}\varphi_1(x_1),\\
    x_1,
      & x_1\in\partial^{d^2}\varphi_1(x_1).
  \end{cases}
\end{equation}
Then $y=T(x_1)$ for $\mu_1$-a.e. $x_1\in\supp\mu_1$.
We refer the reader to \cite{FR10} and \cite[Chapter~3]{Rif14}
for a discussion of optimal transport on sub-Riemannian manifolds.

\begin{proposition}[Exchange property]\label{prop:exchange}
Let $x=(x_1,\ldots,x_m)$ and
$\bar x=(\bar x_1,\ldots,\bar x_m)$ belong to $\Gamma_\varphi$ and share a barycenter $y$. If we
exchange an arbitrary coordinate, then $y$ remains a barycenter. Namely, let
\[
  x'=(x_1,\ldots,\bar x_i,\ldots,x_m),
  \qquad
  \bar x'=(\bar x_1,\ldots,x_i,\ldots,\bar x_m).
\]
Then $y$ is a barycenter of both $x'$ and $\bar x'$.
\end{proposition}

\begin{proof}
Since $\Gamma_\varphi$ is $c$-cyclically monotone and $x,\bar x\in\Gamma_\varphi$, we have
\[
  c(x)+c(\bar x)\le c(x')+c(\bar x').
\]
On the other hand, by the definition of the cost and by using $y$ as a
competitor,
\begin{align*}
  c(x')+c(\bar x')
  &\le d^2(\bar x_i,y)+\sum_{j\ne i}d^2(x_j,y)
       +d^2(x_i,y)+\sum_{j\ne i}d^2(\bar x_j,y)\\
  &=c(x)+c(\bar x).
\end{align*}
Thus equality holds, and $y$ is a barycenter of both exchanged tuples.
\end{proof}

\section{Sub-Riemannian barycenters and the multi-marginal Monge problem}\label{sec:structure}
This section studies sub-Riemannian barycenters and gives a sufficient condition for the existence of a solution to \eqref{mp}.
\subsection{The sub-Riemannian barycenter map}
Assume that $(M,\mathcal D,g)$ is complete and that $d^2$
is locally Lipschitz off the diagonal.
Let $\mu_1\in\mathcal{P}^{\mathrm{ac}}_c(M)$ and $\mu_2,\ldots,\mu_m\in\mathcal{P}_c(M)$, and let $\gamma$ be a solution to \eqref{kp}.
Then the associated \emph{barycenter map} is introduced as follows:
\begin{align}\label{bary1}
    b&\colon\supp\gamma\longrightarrow M\\
    b(x_1,\ldots,x_m)&=y,\qquad y\in\operatorname*{argmin}_{z\in M}\sum_{i=1}^m d^2(x_i,z).
\end{align}
By Proposition~\ref{prop:potential-barycenter} and the subsequent discussion, $b$ is well-defined for $\gamma$-a.e. $x=(x_1,\ldots,x_m)$, and $b(x)=T(x_1)$, where $T$ is defined as in \eqref{eq:barycenter-map}.
In the Riemannian setting, two tuples in the support of an optimal plan $\gamma$ that share a barycenter must coincide; see \cite[Lemma~3.5]{KP15}. In particular, the barycenter map
is injective. This property can fail in the sub-Riemannian setting.
We explain this below.

We first formulate the Lagrange multiplier condition for barycenters on
sub-Riemannian manifolds.

\begin{theorem}\label{thm:lagrange-barycenter}
Let $(M,\mathcal D,g)$ be a complete sub-Riemannian manifold, let
$x=(x_1,\ldots,x_m)$, and let $y$ be a barycenter of $x$. For $i=1,\ldots,m$, let $\gamma_i$ be a minimizing geodesic from $x_i$ to $y$ with control $u_i$. Then there exists $(\eta,\lambda_0)\neq 0$, with
$\eta=(\eta_1,\ldots,\eta_m)\in(T_y^*M)^m$ and
$\lambda_0\in\{1,0\}$ such that, for $i=1,\ldots,m$,
\[
  \begin{cases}
  \eta_i^*D_{u_i}E_{x_i}
    =\lambda_0\displaystyle\int_0^1
      \langle u_i,\cdot\rangle\,dt,\\
  E_{x_i}(u_i)=y.
  \end{cases}
\]
Moreover, 
\begin{equation}\label{balance}
  \displaystyle\sum_{i=1}^m\eta_i=0.
\end{equation}
\end{theorem}
This is the multi-endpoint analog of the standard Lagrange multiplier rule; see \cite[Theorem~B.2]{Rif14}. For completeness,
a proof is deferred to the appendix.

The above characterization of barycenters does not directly provide
the information needed for our purposes. First, the multipliers need not be
unique. Moreover, constant curves are abnormal as well (except in the Riemannian case), so coordinates
coinciding with the barycenter may contribute nonzero covectors.
For example, in the ideal case, when $\lambda_0=1$, each coordinate $x_i\ne y$
satisfies
$x_i=\exp_y(-\eta_i)$,
whereas coordinates with $x_i=y$ may contribute
nonzero covectors in $\Ann(\mathcal D_y)$. Consequently, the balance
condition \eqref{balance} yields only
\[
  \sum_{i:\,x_i\ne y}\eta_i\in\Ann(\mathcal D_y),
\]
rather than the vanishing of this sum.

In the Riemannian setting, barycenters of tuples in $\supp\gamma$
lie outside the cut locus of each coordinate; see
\cite[Lemma~3.1]{KP15}. Therefore, the covectors $\eta_i$ in \eqref{balance} are uniquely
determined by
\[
  \eta_i=\frac12 d_y[d^2(x_i,\cdot)],
\]
and coordinates with $x_i=y$ contribute $\eta_i=0$.
Together with Proposition~\ref{prop:exchange}, these observations show that
each barycenter uniquely determines the corresponding optimal tuple.
Thus, we first map $x_1$ to its barycenter $y$ using
Proposition~\ref{prop:potential-barycenter}, and then recover
$x_2,\ldots,x_m$ from $y$ by injectivity of the barycenter map.
Composing these maps yields a solution to the Monge problem \eqref{mp}.
However, the argument establishing that barycenters lie outside the
cut loci relies on semiconcavity of the squared distance on the diagonal
and therefore does not directly apply in the sub-Riemannian setting. Fortunately, inspired by the work of Ohta \cite[Section~4]{Ohta12}
and Jiang \cite[Theorem~3.5]{Jia17} on Alexandrov spaces,
we observe that it is unnecessary to prove that the barycenter
lies outside the cut locus of each coordinate (see \cite[Remark~3.11]{Jia17}).
We therefore establish the following sub-Riemannian balance condition.

We first introduce some notation.
For a tuple $x=(x_1,\ldots,x_m)$ with a chosen barycenter $y$,
we call the coordinate $x_j$ \emph{moving} if $x_j\ne y$
and \emph{static} if $x_j=y$.
We can now state our characterization of sub-Riemannian barycenters.
\begin{theorem}\label{thm:barycenter-rigidity}
Let $(M,\mathcal D,g)$ be an ideal sub-Riemannian manifold, and let
$x,\bar x\in\supp\gamma$ have a common barycenter $y$.
Then:
\begin{itemize}
  \item[(1)] If every $x_i$ is moving, then $x_i=\bar x_i$ for every $i$.
  \item[(2)] If at least one coordinate $x_i$ is static, then:
  \begin{itemize}
    \item[(a)] for every static coordinate $x_i$, one has
      $x_i=y=\bar x_i$;
    \item[(b)] for each $i$ such that $x_i,\bar x_i\ne y$, let
      $\gamma_i(t)=\exp_y(t\eta_i)$ and
      $\bar\gamma_i(t)=\exp_y(t\bar\eta_i)$, $t\in[0,1]$,
      be minimizing geodesics from $y$ to $x_i$
      and $\bar x_i$, respectively. Then
      \begin{equation}\label{balcond}
        \eta_i-\bar\eta_i\in\Ann(\mathcal D_y).
      \end{equation}
  \end{itemize}
\end{itemize}
\end{theorem}

\begin{proof}
We first establish a balance condition.
For each moving coordinate $x_i$, choose a minimizing
geodesic $\gamma_i:[0,1]\to M$ from $y$ to $x_i$.
By ideality, this geodesic is normal and can be written as
\[
  \dot\gamma_i(t)=\sum_j u^j_i(t)X_j(\gamma_i(t)),
  \qquad
  \gamma_i(0)=y,
  \qquad
  \gamma_i(1)=x_i,
  \qquad
  \gamma_i(t)=\exp_y(t\eta_i),
\]
where $u_i=(u^1_i,\ldots,u^k_i)$ is its control.
Moreover, for each moving coordinate $x_i$, there is a $C^2$
function $F_i$ defined near $y$ such that
\begin{equation}\label{F}
  d^2(x_i,z)\le F_i(z),
  \qquad
  d^2(x_i,y)=F_i(y),
  \qquad
  dF_i(y)=-2\eta_i.
\end{equation}
For the construction of such smooth upper supports, we refer to
the proof of \cite[Theorem~5.9]{FR10}.

Since $y$ minimizes $z\mapsto\sum_j d^2(x_j,z)$, we have
\begin{equation}\label{eq:barycenter-minimum}
  \begin{aligned}
    0
    &\le \sum_j d^2(x_j,\gamma_i(t))-\sum_j d^2(x_j,y)\\
    &=\sum_{j\,\mathrm{moving}}
      \bigl[d^2(x_j,\gamma_i(t))-d^2(x_j,y)\bigr]
      +\sum_{j\,\mathrm{static}}d^2(x_j,\gamma_i(t)).
  \end{aligned}
\end{equation}
Let $N=\#\{j:x_j=y\}$. Then
\begin{equation}\label{eq:static-cost}
  \sum_{j\,\mathrm{static}}d^2(x_j,\gamma_i(t))
  =Nt^2\int_0^1|u_i(s)|^2\,ds
  =Nt^2\|u_i\|_{L^2}^2.
\end{equation}
For each moving coordinate $x_j$, Taylor expansion of the smooth
upper support $F_j$ in \eqref{F} gives
\begin{equation}\label{eq:moving-direction-estimate}
  d^2(x_j,\xi(t))
  \le d^2(x_j,y)-2t\eta_j(\dot\xi(0))+O(t^2),
\end{equation}
where $\xi:[0,\varepsilon]\to M$ is a smooth curve with
$\xi(0)=y$ and $\dot\xi(0)=v\in T_yM$.
In particular,
\begin{equation}\label{eq:moving-geodesic-estimate}
  d^2(x_j,\gamma_i(t))
  \le d^2(x_j,y)-2t\eta_j(\dot\gamma_i(0))+O(t^2).
\end{equation}
Substituting \eqref{eq:static-cost} and
\eqref{eq:moving-geodesic-estimate} into
\eqref{eq:barycenter-minimum}, we obtain
\[
  0\le
  -2t\sum_{j\,\mathrm{moving}}\eta_j(\dot\gamma_i(0))
  +Nt^2\|u_i\|_{L^2}^2+O(t^2).
\]
Dividing by $t>0$ and letting $t\downarrow0$ yields
\[
  \sum_{j\,\mathrm{moving}}\eta_j(\dot\gamma_i(0))
  =\sum_{j\,\mathrm{moving}}
    \langle u_j(0),u_i(0)\rangle
  \le0.
\]
Summing over the moving coordinates $x_i$, we find
\[
  \sum_{i,j\,\mathrm{moving}}
    \langle u_j(0),u_i(0)\rangle
  =\left|\sum_{i\,\mathrm{moving}}u_i(0)\right|^2
  \le0.
\]
Therefore, we obtain the sub-Riemannian balance condition
\begin{equation}\label{bal}
  \sum_{i\,\mathrm{moving}}u_i(0)=0,
  \qquad
  \left.\sum_{i\,\mathrm{moving}}\eta_i\right|_{\mathcal D_y}=0.
\end{equation}
These identities also hold trivially when all coordinates are static.

We next show that $x_i=y$ if and only if $\bar x_i=y$.
Without loss of generality, suppose that $x_1=y$ but $\bar x_1\ne y$.
By Proposition~\ref{prop:exchange}, the tuple obtained from $x$
by replacing $x_1$ with $\bar x_1$ also has barycenter $y$.
Choose a normal minimizing geodesic from $y$ to $\bar x_1$,
with initial covector $\bar\eta_1$ and control $\bar u_1$.
Comparing \eqref{bal} for the two tuples gives
\[
  \bar u_1(0)+\sum_{\substack{i\ge2\\x_i\ne y}}u_i(0)=0
  \quad\Longrightarrow\quad
  \bar u_1(0)=0,
\]
which contradicts $\bar x_1\ne y$.
Interchanging $x$ and $\bar x$ proves the converse.
For the second claim in (2), comparing \eqref{bal} before and
after exchanging the moving coordinates $x_i$ and $\bar x_i$ gives
\begin{equation}
  \eta_i-\bar\eta_i\in\Ann(\mathcal D_y).
\end{equation}

Suppose now that every $x_i$ is moving.
For any $v\in T_yM$, choose a smooth curve $\xi$ with
$\xi(0)=y$ and $\dot\xi(0)=v$.
By \eqref{eq:moving-direction-estimate} and the definition of a barycenter,
\[
  0\le
  \sum_{j=1}^m d^2(x_j,\xi(t))
  -\sum_{j=1}^m d^2(x_j,y)
  \le -2t\sum_{j=1}^m\eta_j(v)+O(t^2).
\]
Dividing by $t>0$, letting $t\downarrow0$, and replacing $v$
by $-v$, we obtain
\[
  \sum_{j=1}^m\eta_j=0.
\]
Since $x$ and $\bar x$ share the barycenter $y$, the preceding argument shows that every $\bar x_i$ is also moving.
For each $i$, Proposition~\ref{prop:exchange} shows that $y$
is also a barycenter of the exchanged tuple
\[
x'=(x_1,\ldots,x_{i-1},\bar x_i,x_{i+1},\ldots,x_m).
\]
Applying the same argument to this tuple gives
\[
  \sum_{j=1}^m\eta_j=0
  =\sum_{j\ne i}\eta_j+\bar\eta_i.
\]
Thus $\eta_i=\bar\eta_i$, and consequently
\[
  x_i=\exp_y(\eta_i)=\exp_y(\bar\eta_i)=\bar x_i.
\]
This proves (1) and completes the proof.
\end{proof}

The preceding theorem suggests a sufficient condition for the existence
of a Monge solution. Namely,
if static coordinates are excluded, each barycenter uniquely determines
the corresponding optimal tuple.
We therefore impose the following condition, which is condition~(C2)
in \cite{PPV21}. Under this condition, we solve \eqref{mp} assuming absolute continuity only of the first marginal.
\begin{definition}[Condition A]\label{condA}
Let
\begin{equation}\label{setB}
    B:=\left\{y\in M:
    y\text{ is a barycenter of some }
    (x_1,\ldots,x_m)\in\prod_{i=1}^m\supp\mu_i\right\}.
\end{equation}
We say that $\mu_1,\ldots,\mu_m$ satisfy \emph{Condition A} if
\[
  \mu_i(B)=0,\qquad i=1,\ldots,m.
\]
\end{definition}

\begin{theorem}\label{mppp}
Let $(M,\mathcal D,g)$ be an ideal sub-Riemannian manifold.
Assume that $\mu_1\in\Pac(M)$ and $\mu_i\in\Pc(M)$ for $i=2,\ldots,m$. If these marginals satisfy Condition A, then
the Monge problem \eqref{mp} admits a solution.
\end{theorem}
\begin{proof}
Let $\gamma$ be a solution to \eqref{kp}.
By Condition A, $\gamma$-a.e. tuple has no static coordinates.
By Proposition~\ref{prop:potential-barycenter}, its barycenter
is uniquely determined by $x_1$ and is given by $y=T(x_1)$,
where $T$ is defined in \eqref{eq:barycenter-map}.
On the other hand, Theorem~\ref{thm:barycenter-rigidity}
implies that this barycenter uniquely determines the tuple.
Thus, for $\mu_1$-a.e. $x_1$, the remaining coordinates
$x_2,\ldots,x_m$ are uniquely determined on this plan, which
is therefore induced by a map and gives a solution to \eqref{mp}.
\end{proof}

\subsection{Counterexamples}
We have derived the sub-Riemannian balance condition and used it to partially solve \eqref{mp} under the ideality assumption. Two questions naturally arise. First, can the balance condition \eqref{balcond} be strengthened to establish a one-to-one correspondence between barycenters and optimal tuples? 
Second, in the proof of Theorem~\ref{thm:barycenter-rigidity},
ideality ensures that every nonconstant minimizing geodesic is
non-abnormal and that the squared distance is locally semiconcave
off the diagonal. 
The absence of abnormality allows us to construct smooth
upper supports \eqref{F}. 
Can this regularity assumption be weakened, for instance, to local Lipschitz continuity off the diagonal? The following two counterexamples give negative answers to these questions in general.

The first example shows that a barycenter need not determine its tuple.

\begin{theorem}\label{thm:general-counterexample}
Let $(M,\mathcal D,g)$ be an ideal sub-Riemannian manifold with
$\mathcal D\ne TM$. Let $x,o,y$ be three points in $M$ such that $x\ne y$ are joined by a 
minimizing geodesic $\gamma:[-1,1]\to M$, with
\[
 \gamma(-1)=x,\qquad \gamma(0)=o,\qquad \gamma(1)=y,
 \qquad x\notin\Cut(y).
\]
Assume further that $\gamma(t)=\exp_o(-t\eta)$ for some $\eta\in T_o^*M$ and all $t\in[-1,1]$.
Then, for arbitrary fixed nonzero $\alpha\in\Ann(\mathcal D_o)$, there exists a
sufficiently small $\varepsilon>0$ such that the point
$x_\varepsilon=\exp_o(\eta+\varepsilon\alpha)$
is distinct from $x$, and $o$ is the unique barycenter of
$(x_\varepsilon,y,o)$, as well as of $(x,y,o)$.
\end{theorem}

\begin{proof}
Fix a nonzero $\alpha\in\operatorname{Ann}(\mathcal D_o)$. Set
\[
 a=d(x,o)=d(y,o)=\tfrac12d(x,y).
\]
Since $\gamma$ is a minimizing curve from $x$ to $y$ and $x\notin \Cut(y)$, $\exp_o$ is a diffeomorphism from a neighborhood of $\eta$ onto a neighborhood of $x$. We can therefore choose $\varepsilon$ sufficiently small so that the normal geodesic with initial covector
$\eta+\varepsilon\alpha$ remains minimizing and connects $o$ to $x_\varepsilon=\exp_o(\eta+\varepsilon\alpha)$, which is distinct from $x$. Since
$H(o,\eta+\varepsilon\alpha)=H(o,\eta),$
the length of this geodesic is unchanged, and hence
\begin{equation}\label{eq:ce-equal-length}
 d(x_\varepsilon,o)=d(x,o)=a.
\end{equation}

Define
\begin{equation}\label{de}
  \widetilde F(p)=d^2(x_\varepsilon,p)+d^2(o,p)+d^2(y,p),\qquad
 F(p)=d^2(x,p)+d^2(o,p)+d^2(y,p).
\end{equation}
 
The triangle inequality gives
\begin{align}\label{eq}
 F(p)&\ge \tfrac12\bigl(d(x,p)+d(y,p)\bigr)^2+d^2(o,p)\notag\\
 &\ge \tfrac12d^2(x,y)+d^2(o,p)
   =F(o)+d^2(o,p).
\end{align}
We will prove that $o$ is also the unique minimizer of $\widetilde F$. We consider points near $o$ and points away from $o$ separately. Near $o$, we will see how the sub-Riemannian geometry, through the Ball--Box estimate, gives rise to this pathological phenomenon.

\medskip\noindent\textbf{Step 1. Points near $o$.}
Let $B(o,r):=\{p\in M:d(o,p)<r\}$. By
Theorem~\ref{thm:ball-box}, for sufficiently small $r<1$,
there exists a coordinate chart centered at $o$ in which
\begin{equation}\label{eq:ce-box}
  2|\alpha(p)|\le C_\alpha d^2(o,p),
  \qquad |p|^2\le C_0d^2(o,p),
  \qquad p\in B(o,r).
\end{equation}
Here $\alpha$ is regarded as a linear function in these coordinates and $p=(z_1,\ldots,z_n)$. The constants $C_\alpha,C_0$ are fixed on this neighborhood and independent of
$\varepsilon$ and $p$.

Set
\[
 G_\varepsilon(p)=d^2(x_\varepsilon,p)+d^2(y,p).
\]
Then $(\varepsilon,p)\mapsto G_\varepsilon(p)$ is smooth near $(0,o)$. Since $o$ is the midpoint of the minimizing geodesic from $x$ to $y$, we have $G_0(p)\ge\tfrac12d^2(x,y)=G_0(o)$. Consequently, 
\[
 DG_0(o)=0,\qquad D^2G_0(o)\ge0.
\]
Fix $\kappa>0$ so small that $\kappa C_0/2<1/4$.
By continuity of the Hessian at $(0,o)$, there exist $\varepsilon_0>0$ and
a convex neighborhood $V$ of $o$ such that
\begin{equation}\label{eq:ce-hessian}
 D_z^2G_\varepsilon(z)[v,v]\ge-\kappa|v|^2
 \quad\text{for } |\varepsilon|<\varepsilon_0,\ z\in V,\ v\in\R^n.
\end{equation}
On the other hand, by direct computation,
\[
 D_zG_\varepsilon(o)
 =-2(\eta+\varepsilon\alpha)+2\eta=-2\varepsilon\alpha,
\]
and Taylor's formula yields
\begin{align*}
 G_\varepsilon(p)-G_\varepsilon(o)\ge-2\varepsilon|\alpha(p)|-\tfrac12\kappa|p|^2.
\end{align*}
Combining this with \eqref{de}, we obtain
\begin{align*}
 \widetilde F(p)-\widetilde F(o)
 &=G_\varepsilon(p)-G_\varepsilon(o)+d^2(o,p)\\
 &\ge\left(1-\varepsilon C_\alpha-\tfrac12\kappa C_0\right)d^2(o,p).
\end{align*}
Choose $r>0$ so that $B(o,r)$ lies in  $V$, and then choose
$0<\varepsilon<\varepsilon_0$ so small that $\varepsilon C_\alpha<1/4$.
It follows that
\[
 \widetilde F(p)-\widetilde F(o)\ge\tfrac12d^2(o,p)>0
 \qquad (0<d(o,p)<r).
\]

\medskip\noindent\textbf{Step 2. Points outside $B(o,r)$.}
Set $s=d(o,p)\ge r$ and $\delta=d(x_\varepsilon,x)$. By
\eqref{eq} and the triangle inequality,
\begin{align*}
 \widetilde F(p)-\widetilde F(o)
 &=F(p)-F(o)+d^2(x_\varepsilon,p)-d^2(x,p)\\
 &\ge s^2-|d(x_\varepsilon,p)-d(x,p)|\bigl(d(x_\varepsilon,p)+d(x,p)\bigr)\\
 &\ge s^2-\delta\bigl(d(x_\varepsilon,o)+d(x,o)+2d(o,p)\bigr)\\
 &=s^2-\delta(2a+2s).
\end{align*}
The function $s\mapsto s^2/(2a+2s)$ is strictly increasing for $s>0$.
Consequently, it suffices to choose $\delta$ such that
\[
 \delta<\frac{r^2}{2a+2r}
       =\frac{r^2}{d(x,y)+2r},
\]
which ensures $\widetilde F(p)>\widetilde F(o)$ for every $s\ge r$.
This condition is achieved by further decreasing the value of $\varepsilon$ chosen above, because
$x_\varepsilon\to x$ as $\varepsilon\downarrow 0$.

Combining the estimates above, we can choose $\varepsilon$ sufficiently small so that
$\widetilde F(p)>\widetilde F(o)$ for every $p\ne o$. Hence $o$ is the
unique global barycenter of $(x_\varepsilon,y,o)$, with $x_\varepsilon\ne x$.
\end{proof}

As a basic model of sub-Riemannian geometry, the Heisenberg group plays an important role in our study. We recall its definition. On $\mathbb{R}^3$, let $\mathcal D$ be spanned by
\[
X=\frac{\partial}{\partial x}-\frac{y}{2}\frac{\partial}{\partial z},\qquad Y=\frac{\partial}{\partial y}+\frac{x}{2}\frac{\partial}{\partial z},
\]
and let $g$ be the metric on $\mathcal D$ for which $X,Y$ are orthonormal. Then $\mathbb H=(\mathbb R^3,\mathcal D,g)$ is the Heisenberg group. This structure is ideal. Choose a minimizing geodesic $\gamma$ passing through the origin $O$, with endpoints $x,y$ and midpoint $O$, such that $x\notin \Cut(y)$. By Theorem~\ref{thm:general-counterexample}, there exists $x'$ close to $x$ such that $O$ is the unique barycenter of both $(x,y,O)$ and $(x',y,O)$. The following example shows that local Lipschitz continuity alone is insufficient to derive Theorem~\ref{thm:barycenter-rigidity}.

\begin{example}
With the above choices of $x,x',y,O$, consider $\mathbb H\times\mathbb R$ equipped with the product sub-Riemannian structure. Its distance is given by
\begin{equation}
d^2((q,s),(q',s'))=d_{\mathbb H}^2(q,q')+|s-s'|^2,
\end{equation}
where $d_{\mathbb H}$ denotes the sub-Riemannian distance on $\mathbb H$. The squared distance $d^2$ is locally Lipschitz in local coordinates.  However, it is not locally semiconcave off the diagonal; see the related discussion in \cite[Section~4.2]{BR19} and \cite{MM}. Consider the two distinct tuples $\big((x,1),(y,1),(O,-2)\big)$ and $\big((x',1),(y,1),(O,-2)\big)$. One can check that both have the unique barycenter $(O,0)$. Every coordinate of either tuple is distinct from $(O,0)$. Thus, without semiconcavity, local Lipschitz regularity of the distance off the diagonal does not guarantee injectivity of the barycenter map, even when all coordinates are moving.
\end{example}

\section{Absolute continuity of the Wasserstein barycenter}\label{sec:ac}
We now turn to Wasserstein barycenters, defined as follows.
\begin{definition}[Wasserstein barycenter]
Let $\mu_1,\ldots,\mu_m\in\Pc(M)$.
A probability measure $\nu\in\mathcal P(M)$ is called a
\emph{Wasserstein barycenter} of $\mu_1,\ldots,\mu_m$ if it minimizes
$\sigma\longmapsto\sum_{i=1}^m W_2^2(\mu_i,\sigma)$
over $\sigma\in\mathcal P(M)$, where
\[
  W_2^2(\mu,\sigma)
  :=\inf_{\pi\in\Pi(\mu,\sigma)}
    \int_{M\times M}d^2(x,y)\,d\pi(x,y).
\]
Here $\Pi(\mu,\sigma)$ denotes the set of couplings of
$\mu$ and $\sigma$.
\end{definition}

On an ideal sub-Riemannian manifold, the argument in
\cite[Theorem~3.1]{KP17} also establishes the existence of a Wasserstein barycenter. 
If, in addition, $\mu_1\in\mathcal P_c^{\mathrm{ac}}(M)$, uniqueness follows from the strict convexity of $\sigma\mapsto \sum\limits_i W_2^2(\mu_i,\sigma)$.
Moreover, for any optimal plan $\gamma$ and associated barycenter
map $b$ as in \eqref{bary1}, 
\[
  \nu=b_\sharp\gamma=T_\sharp\mu_1,
\]
where $T$ is the unique optimal transport from $\mu_1$ to $\nu,$ given in 
\eqref{eq:barycenter-map}.

The absolute continuity of Wasserstein barycenters in Riemannian settings, which requires only one absolutely continuous marginal, is proved in \cite[Theorem~5.1]{KP17}. We follow the same general strategy. After fixing $x_i\in\supp\mu_i$ for $i=2,\ldots,m$, we regard the barycenter map as a function of the first coordinate. We seek a locally Lipschitz inverse whose Lipschitz constant is uniform in the fixed coordinates. In the Riemannian setting, the required uniform bounds follow from geometric comparison estimates. In the sub-Riemannian setting, a more delicate local argument is needed. The remaining approximation strategy follows the Riemannian argument.

\begin{theorem}\label{thm:ac-barycenter-draft}
Let $(M,\mathcal D,g)$ be an ideal sub-Riemannian manifold.
Assume that $\mu_i\in\mathcal P_c^{\mathrm{ac}}(M)$ for $i=1,\ldots,m$.
Then the Wasserstein barycenter $\nu$ of
$\mu_1,\ldots,\mu_m$ is absolutely continuous with respect
to $\mathfrak m$.
\end{theorem}

\begin{proof}
\textbf{Step 0. Setup.}
Since all marginals have compact support, there exist
$p_0\in M$ and $R>0$ such that
\[
  \bigcup_{i=1}^m\supp\mu_i\subset\overline{B(p_0,R)}.
\]
The triangle inequality implies that every barycenter of a tuple
in $\prod_{i=1}^m\supp\mu_i$ lies in $\overline{B(p_0,2R)}$.

Fix $\ell\in \mathbb{N}^+$ and $\hat x=(x_2,\ldots,x_m)\in\prod_{i=2}^m\supp\mu_i$. Let $\nu^0$ be the Wasserstein barycenter of
$\mu_1,\delta_{x_2},\ldots,\delta_{x_m}$, and let $T^0$ be the optimal
transport map from $\mu_1$ to $\nu^0$. By Proposition~\ref{prop:potential-barycenter} and the subsequent discussion,
$T^0$ is defined on a set $\Omega^0\subset\supp\mu_1$ of full $\mu_1$-measure.
Define
\[
  A_{\hat x,\ell}
  :=\{x_1\in\Omega^0:\min_{i\in \{1,\ldots,m\}} d(x_i,T^0(x_1))>\frac1\ell\}.
\]
For the fixed tuple $\hat x$, introduce the notation
\begin{equation}\label{eq:want}
    \mathfrak{bc}_{\hat x}(x_1):=T^0(x_1),
    \qquad x_1\in A_{\hat x,\ell}.
\end{equation}
The notation $T^0(x_1)$ emphasizes the optimal transport map, whereas
$\mathfrak{bc}_{\hat x}(x_1)$ emphasizes its interpretation as the barycenter of
$(x_1,x_2,\ldots,x_m)$.
Since all coordinates of these tuples are moving, Theorem~\ref{thm:barycenter-rigidity} implies that $\mathfrak{bc}_{\hat x}$ is injective.
We show that its inverse is locally Lipschitz, with Lipschitz
constants independent of $\hat x$.

\textbf{Step 1. Uniform Lipschitz estimate.}
To this end, choose finitely many relatively compact coordinate
neighborhoods $U_{\iota}$, convex in their coordinates, such that
$\overline{U}_{\iota}\subset B(p_{\iota},1/(4\ell))$ and the $U_{\iota}$ cover
$\overline{B(p_0,2R)}$.
For each $\iota$, set
\[
  A_{\iota,\ell}
  :=\left\{x_1\in\supp\mu_1:
      d(x_1,p_{\iota})>\frac3{4\ell}\right\}.
\]
For $x_1\in A_{\hat x,\ell}$ with $\mathfrak{bc}_{\hat x}(x_1)\in U_{\iota}$,
the triangle inequality gives $x_1\in A_{\iota,\ell}$. This construction allows us to work on the sets $A_{\iota,\ell}$,
which are independent of $\hat x$, and thus obtain estimates
uniform in $\hat x$.
Meanwhile,
\[
  d(x_1,z)>\frac1{2\ell},
  \qquad x_1\in A_{\iota,\ell},\quad z\in U_{\iota}.
\]
Thus, by local semiconcavity of the squared distance on the
relatively compact set $A_{\iota,\ell}\times U_{\iota}$, the functions
$z\mapsto d^2(x_1,z)$ have a common semiconcavity constant
on $U_{\iota}$, independent of $x_1\in A_{\iota,\ell}$.
Moreover, for those indices $\iota$ with
$\mathfrak{bc}_{\hat x}(A_{\hat x,\ell})\cap U_{\iota}\ne\varnothing$,
the triangle inequality gives
\[
  d(x_i,z)>\frac1{2\ell},
  \qquad i=2,\ldots,m,\quad z\in U_{\iota}.
\]
Consequently, the function
\[
  g_{\hat x,\iota,\ell}(z):=\sum_{i=2}^m d^2(x_i,z),
  \qquad z\in U_{\iota},
\]
also has a semiconcavity constant independent of $\hat x$. Since the cover is finite, this constant can also be chosen independently of $\iota$.
Define
\[
  \alpha_{\hat x,\iota,\ell}(x_1)
  :=\inf_{w\in U_{\iota}}
    \{d^2(x_1,w)+g_{\hat x,\iota,\ell}(w)\},
  \qquad x_1\in A_{\iota,\ell},
\]
and
\[
  \beta_{\hat x,\iota,\ell}(z)
  :=\inf_{x_1\in A_{\iota,\ell}}
    \{d^2(x_1,z)-\alpha_{\hat x,\iota,\ell}(x_1)\},
  \qquad z\in U_{\iota}.
\]
As an infimum of semiconcave functions with a common
semiconcavity constant, $\beta_{\hat x,\iota,\ell}$ is semiconcave as well.
We claim that
\[
  g_{\hat x,\iota,\ell}(z)+\beta_{\hat x,\iota,\ell}(z)\ge0
  \qquad\text{for every }z\in U_{\iota},
\]
with equality on $\mathfrak{bc}_{\hat x}(A_{\hat x,\ell})\cap U_{\iota}$.
Indeed, the definition of $\alpha_{\hat x,\iota,\ell}$ gives
\[
  -g_{\hat x,\iota,\ell}(z)
  \le d^2(x_1,z)-\alpha_{\hat x,\iota,\ell}(x_1)
\]
for every $x_1\in A_{\iota,\ell}$ and $z\in U_{\iota}$.
Taking the infimum over $x_1$ proves the inequality.
If $x_1\in A_{\hat x,\ell}$ and $y=\mathfrak{bc}_{\hat x}(x_1)\in U_{\iota}$,
then $x_1\in A_{\iota,\ell}$ and
\[
  \alpha_{\hat x,\iota,\ell}(x_1)
  =d^2(x_1,y)+g_{\hat x,\iota,\ell}(y).
\]
It follows that
\[
  \beta_{\hat x,\iota,\ell}(y)
  \le d^2(x_1,y)-\alpha_{\hat x,\iota,\ell}(x_1)
  =-g_{\hat x,\iota,\ell}(y),
\]
which proves equality.

Because $-g_{\hat x,\iota,\ell}$ is semiconvex and $\beta_{\hat x,\iota,\ell}$ is semiconcave, \cite[Theorem~A.19]{FF10} implies that both functions are
differentiable on $\mathfrak{bc}_{\hat x}(A_{\hat x,\ell})\cap U_{\iota}$,
where $d\beta_{\hat x,\iota,\ell}=-dg_{\hat x,\iota,\ell}$.
Moreover, the same theorem yields these differentials are locally Lipschitz on this set,
with uniform local Lipschitz constants determined by the
common semiconcavity constants of $g_{\hat x,\iota,\ell}$.
For $x_1\in A_{\hat x,\ell}$ with $y=\mathfrak{bc}_{\hat x}(x_1)\in U_{\iota}$,
we also have
\[
  \beta_{\hat x,\iota,\ell}(z)
  \le d^2(x_1,z)-\alpha_{\hat x,\iota,\ell}(x_1),
  \qquad z\in U_{\iota},
\]
with equality at $z=y$.
Proposition~\ref{prop:lower-support} therefore gives
\begin{equation}\label{afun}
    x_1=\exp_y\left(-\frac12d\beta_{\hat x,\iota,\ell}(y)\right)
    =:\mathfrak a_{\hat x,\iota,\ell}(y).
\end{equation}
Thus
\[
  \mathfrak a_{\hat x,\iota,\ell}(\mathfrak{bc}_{\hat x}(x_1))=x_1
  \qquad
  \text{whenever }x_1\in A_{\hat x,\ell}
  \text{ and }\mathfrak{bc}_{\hat x}(x_1)\in U_{\iota}.
\]
These local formulas agree on overlaps, since they all recover $x_1$.
Therefore we can define the inverse
\[
  \mathfrak a_{\hat x,\ell}:\mathfrak{bc}_{\hat x}(A_{\hat x,\ell})
  \longrightarrow A_{\hat x,\ell}.
\]
The representation \eqref{afun} shows that $\mathfrak a_{\hat x,\ell}$ is locally Lipschitz, uniformly in $\hat x$.
Therefore, for every Borel set $E\subset M$,
\begin{equation}\label{eq:slice-volume-bound}
  \mathfrak m\left(\mathfrak a_{\hat x,\ell}
    (E\cap \mathfrak{bc}_{\hat x}(A_{\hat x,\ell}))\right)
  \le C_\ell\mathfrak m
    (E\cap \mathfrak{bc}_{\hat x}(A_{\hat x,\ell})),
\end{equation}
where $C_\ell$ is a uniform constant with respect to $\hat x$.

\textbf{Step 2. Discrete approximation argument.}
We now apply a discrete approximation argument, following
\cite[Theorem~5.1, Lemmas~5.2--5.4]{KP17}.
For $i=2,\ldots,m$, we choose finitely supported probability measures
$\mu_i^N\rightharpoonup\mu_i$ with
$\supp\mu_i^N\subset\supp\mu_i$ (here we use ``$\rightharpoonup$'' to denote weak-$*$ convergence).
Let $\gamma^N$ be an optimal plan for
$\mu_1,\mu_2^N,\ldots,\mu_m^N$, and let $b^N$ be its barycenter map.
Let $T^N$, defined on $\Omega^N\subset\supp\mu_1$, be the optimal transport map
from $\mu_1$ to $\nu^N:=(b^N)_\sharp\gamma^N$.
Here,
we choose $\Omega^N$ as a set of full $\mu_1$-measure so that every $x\in\supp\gamma^N$ with
$x_1\in\Omega^N$ has the unique barycenter $b^N(x)=T^N(x_1)$.
The set $\overline{\Omega}:=\bigcap_{N=1}^\infty \Omega^N$ also has full $\mu_1$-measure.
By \cite[Lemma~5.4]{KP17},
$\nu^N\rightharpoonup\nu.$
Define the joint measure
\[
  \rho^N:=(\mathrm{id},b^N)_\sharp\gamma^N.
\]
After passing to a subsequence, we have $\rho^N\rightharpoonup\rho$,
where the first $m$ marginals of $\rho$ are $\mu_1,\ldots,\mu_m$ and its last marginal is $\nu$.
For $\hat x\in \prod_{i=2}^m\supp\mu_i^N$, define
\[
  A^N_{\hat x,\ell}
  :=\left\{x_1\in\overline{\Omega}:
        (x_1,\hat x)\in\supp\gamma^N,\quad
        \min_{i\in \{1,\ldots,m\}} d(x_i,T^N(x_1))>\frac1\ell\right\}.
\]
Here we allow $A^N_{\hat x,\ell}$ to be empty.
Set
\[
  \mathfrak{bc}^N_{\hat x}(x_1):=b^N(x_1,\hat x)=T^N(x_1),
  \qquad x_1\in A^N_{\hat x,\ell}.
\]
By Theorem~\ref{thm:barycenter-rigidity}(1), since all coordinates are moving
and the barycenter is determined by $x_1$, the tuple $\hat x$ is uniquely determined by $x_1$.
The same theorem shows that, for each $N$, the sets
$\mathfrak{bc}^N_{\hat x}(A^N_{\hat x,\ell})$, $\hat x\in \prod_{i=2}^m\supp\mu_i^N$, are pairwise
disjoint. We use the same notation $\mathfrak a_{\hat x,\ell}$
for the inverse of $\mathfrak{bc}^N_{\hat x}$ on its image.
Repeating the argument of \textbf{Step 1} on each of these sets shows that
$\mathfrak a_{\hat x,\ell}$ is locally Lipschitz, with the same uniform
local Lipschitz constant, independent of $\hat x$ and $N$.

Suppose, for a contradiction, that a Borel set $S\subset M$ satisfies
$\mathfrak m(S)=0$ and $\nu(S)>\delta>0$.
Since every $\mu_i$ is absolutely continuous,
\begin{equation}\label{eq:static}
  \begin{aligned}
   &\rho(\{(x_1,\ldots,x_m,y):y\in S,\min_{i\in \{1,\ldots,m\}} d(y,x_i)=0\})\\
   &\qquad\le\sum_{i=1}^m
         \rho(\{(x_1,\ldots,x_m,y):x_i\in S\})
     =\sum_{i=1}^m\mu_i(S)=0.
  \end{aligned}
\end{equation}
Consequently,
\[
  \rho(\{(x_1,\ldots,x_m,y):y\in S,\ \min_{i\in \{1,\ldots,m\}} d(x_i,y)>0\})
  =\nu(S)>\delta.
\]
We choose open sets $E_k\supset S$ with
$\mathfrak m(E_k)\le2^{-k}$, and define the open sets
\[
  O_{k,\ell}:=\left\{(x_1,\ldots,x_m,y)\in M^m\times E_k:
                     \min_{i\in \{1,\ldots,m\}} d(x_i,y)>\frac1\ell\right\}.
\]
By continuity from below, there exists a fixed $\ell\in\mathbb N$ such that $\rho(O_{k,\ell})>\delta/2$ for every $k$. 

By weak-$*$ convergence, for each $k$, there exists a sufficiently large $N=N(k)$ such that
\begin{align}\label{con}
  \frac\delta3<\rho^{N}(O_{k,\ell})
  &=\gamma^{N}\left(\left\{x:b^N(x)\in E_k,
                   \ \min_i d(x_i,b^N(x))>\frac1\ell\right\}\right)\\
  &\le\mu_1\left(
       \bigcup_{\hat x\in \prod_{i=2}^m\supp\mu_i^N}
       \mathfrak a_{\hat x,\ell}
       (E_k\cap \mathfrak{bc}^N_{\hat x}(A^N_{\hat x,\ell}))\right).
\end{align}
On the other hand, by \eqref{eq:slice-volume-bound},
\begin{align*}
 &\mathfrak m\left(
       \bigcup_{\hat x\in \prod_{i=2}^m\supp\mu_i^N}
       \mathfrak a_{\hat x,\ell}
       (E_k\cap \mathfrak{bc}^N_{\hat x}(A^N_{\hat x,\ell}))\right)\\
 &\qquad\le C_\ell\sum_{\hat x\in \prod_{i=2}^m\supp\mu_i^N}
       \mathfrak m(E_k\cap \mathfrak{bc}^N_{\hat x}(A^N_{\hat x,\ell}))
 \le C_\ell\mathfrak m(E_k)\le C_\ell2^{-k}.
\end{align*}
By the absolute continuity of $\mu_1$, choosing $k$ sufficiently large yields a contradiction to \eqref{con}. This completes the proof.
\end{proof}

The absolute continuity of $\mu_i$, $i\ge2$, is used only in
\eqref{eq:static} to handle static coordinates.
When static coordinates are present, Theorem~\ref{thm:barycenter-rigidity} no longer
guarantees the injectivity of the map $\mathfrak{bc}_{\hat x}$
constructed in \eqref{eq:want}, let alone the existence of a
locally Lipschitz inverse.
However, Theorem~\ref{thm:barycenter-rigidity}(2) still implies
that the preimage of a fixed barycenter is contained in a
lower-dimensional subset of $M$ and therefore has zero
$\mathfrak m$-measure. Due to this observation, we require that, for each marginal,
the singular mass at static points be concentrated on an at most
countable set, which is weaker than both
Condition~A (Definition~\ref{condA}) and the absolute continuity of all marginals.
\begin{definition}[Condition B]\label{condAA}
Let
\begin{equation}
    B:=\left\{y\in M:
    y\text{ is a barycenter of some }
    (x_1,\ldots,x_m)\in\prod_{i=1}^m\supp\mu_i\right\}.
\end{equation}
We say that $\mu_1,\ldots,\mu_m$ satisfy \emph{Condition B}
if there exists an at most countable set $C\subset B$ such that,
for every Borel set $S\subset M$ with $\mathfrak m(S)=0$,
\[
  \mu_i((S\cap B)\setminus C)=0,\qquad i=1,\ldots,m.
\]
\end{definition}

\begin{theorem}\label{54}
Let $(M,\mathcal D,g)$ be an ideal sub-Riemannian manifold, with
$\mu_1\in\mathcal P_c^{\mathrm{ac}}(M)$ and
$\mu_i\in\mathcal P_c(M)$ for $i=2,\ldots,m$, satisfying Condition B.
Then the Wasserstein barycenter $\nu$ of
$\mu_1,\ldots,\mu_m$ is absolutely continuous with respect to
$\mathfrak m$.
\end{theorem}

\begin{proof}
We use the notation of the proof of
Theorem~\ref{thm:ac-barycenter-draft}.
Let $\rho$ be the weak-* limit obtained there from $\rho^N$,
whose projection onto the first $m$ coordinates $(\pi_1,\ldots,\pi_m)_\sharp\rho$ is an optimal plan.
Fix $y$ such that $(x,y)\in\supp\rho$ and
$\min\limits_{i\in \{1,\ldots,m\}} d(x_i,y)=0$ for some $x$.
By Theorem~\ref{thm:barycenter-rigidity}(2), there exists
$\eta\in T_y^*M$ such that
\[
  \{x_1:(x_1,\ldots,x_m,y)\in\supp\rho\}
  \subseteq
  \exp_y\bigl(\eta+\Ann(\mathcal D_y)\bigr)\cup\{y\}
\]
(if $x_1=y$ we can choose $\eta=0$).
Since $\exp_y$ is smooth and
\[
  \dim\bigl(\eta+\Ann(\mathcal D_y)\bigr)
  =n-\dim\mathcal D_y<n,
\]
we obtain
\[
  \begin{aligned}
    \nu(\{y\})
    &=\rho(\{(x_1,\ldots,x_m,y'):y'=y\})\\
    &\le\mu_1\!\left(
      \exp_y\bigl(\eta+\Ann(\mathcal D_y)\bigr)\cup\{y\}
      \right)=0.
  \end{aligned}
\]
Writing the countable set $C$ in Definition~\ref{condAA} as $\{y_j\}_j$, it follows that
\[
  \begin{aligned}
    &\rho(\{(x,y):y\in C,\ \min_{i\in \{1,\ldots,m\}} d(x_i,y)=0\})\\
    &\quad=
      \sum_j\rho(\{(x,y):y=y_j,\ \min_{i\in \{1,\ldots,m\}} d(x_i,y)=0\})
      =0.
  \end{aligned}
\]
For any Borel set $S\subset M$ with $\mathfrak m(S)=0$,
we therefore have
\[
  \begin{aligned}
    &\rho(\{(x,y):y\in S,\ \min_{i\in \{1,\ldots,m\}} d(x_i,y)=0\})\\
    &\quad\le
      \rho(\{(x,y):y\in C,\ \min_{i\in \{1,\ldots,m\}} d(x_i,y)=0\})
      +\mu_1(S)\\
    &\qquad\quad+
      \sum_{i=2}^{m}
      \rho(\{(x,y):y\in S\setminus C,\ y=x_i\})
      =0,
  \end{aligned}
\]
where the last sum vanishes by Condition B.
This replaces \eqref{eq:static}; the rest of the proof is unchanged.
\end{proof}

By the absolute continuity established above, the classical
transport theorem yields a Monge representation of the MMOT optimizer.

\begin{theorem}\label{thm:monge-final}
Let $(M,\mathcal D,g)$ be an ideal sub-Riemannian manifold, with
$\mu_1\in\mathcal P_c^{\mathrm{ac}}(M)$ and
$\mu_i\in\mathcal P_c(M)$ for $i=2,\ldots,m$, satisfying Condition B.
Then the Monge problem \eqref{mp} admits a unique solution,
up to $\mu_1$-almost everywhere equality.
\end{theorem}

\begin{proof}
Let $\gamma$ be any optimal plan for the multi-marginal problem,
and let $b$ be its barycenter map.
Then 
the plans $(\pi_1,b)_\sharp\gamma$ and $(b,\pi_i)_\sharp\gamma$
are optimal with marginals $(\mu_1,\nu)$ and $(\nu,\mu_i)$,
respectively, for $i=2,\ldots,m$.
Since $\mu_1,\nu\ll\mathfrak m$,
\cite[Theorem~3.2]{FR10} gives unique optimal maps
$T$ from $\mu_1$ to $\nu$ and $T_i$ from $\nu$ to $\mu_i$.
Hence
$\gamma=(\mathrm{id},T_2\circ T,\ldots,T_m\circ T)_\sharp\mu_1$ and $F=(T_2\circ T,\ldots,T_m\circ T)$ solves \eqref{mp}.
Since $\nu$ and these optimal maps are unique, $F$ is unique.
\end{proof}
As a direct corollary, since each coordinate of an optimal tuple can be obtained via optimal transport from the Wasserstein barycenter to the corresponding marginal, \cite[Corollary~38]{BR19} implies that, for $\gamma$-almost every $x=(x_1,\ldots,x_m)$, the corresponding barycenter $y$ lies outside $\Cut(x_i)$ whenever $x_i\neq y$. This extends \cite[Lemma~3.1]{KP15}.

\begin{corollary}
  Under the hypotheses of Theorem~\ref{thm:monge-final}, for $\gamma$-a.e. $x=(x_1,\ldots,x_m)$, the corresponding barycenter $y\notin\Cut(x_i)$ for every $i$ such that $x_i\neq y$.
\end{corollary}

\section{Further discussion}\label{further-discussion}
We end this article by discussing two further questions.

\textbf{Topic 1.} For $p>1$, define the $p$-barycentric cost by
\[
c_p(x_1,\ldots,x_m)
:=\inf_{y\in M}\frac1p\sum_{i=1}^m d^p(x_i,y).
\]
Given probability measures $\mu_1,\ldots,\mu_m\in\mathcal P_c(M)$,
the corresponding $p$-multi-marginal Monge problem is
\[
\inf_{\substack{
                (T_i)_\sharp\mu_1=\mu_i,\ i=2,\ldots,m}}
\int_M c_p\bigl(x,T_2(x),\ldots,T_m(x)\bigr)\,d\mu_1(x).
\]

Assume that the sub-Riemannian structure has step $s$.
If $p>s$, the Ball--Box theorem implies that, for every $y\in M$,
the function $x\mapsto d^p(y,x)$ is differentiable at $x=y$,
with zero differential. Thus static coordinates make no contribution to
the balance condition. The same observation was made in the
Heisenberg setting by Pass, Pinamonti, and Vedovato
\cite[Section~4.1]{PPV21}. This suggests that a stronger balance condition may hold,
leading to existence and uniqueness for the $p$-multi-marginal
Monge problem under weaker assumptions.
We hope to obtain such results in a more general
sub-Riemannian setting.

\textbf{Topic 2.} Using multi-marginal optimal transport, Kim and Pass established a generalized Brunn--Minkowski inequality \cite[Theorem~8.1]{KP17}. Han, Liu and Zhu extended these results to RCD spaces and introduced a new synthetic curvature-dimension condition, called the barycenter curvature-dimension (BCD) condition \cite[Definition~6.1]{HLZ24}. Recent work has shown that the BCD condition can detect infinitesimal Hilbertianity. In particular, while non-Riemannian Finsler manifolds may satisfy the CD condition, smooth reversible non-Riemannian Finsler manifolds cannot satisfy the corresponding BCD condition \cite{HL26a,HL26b}. Although genuine sub-Riemannian manifolds generally fail to satisfy the classical CD condition, it would be interesting to establish suitable barycentric inequalities in this setting, such as Jensen-type inequalities and multi-marginal extensions of the two-marginal interpolation inequalities in \cite{BR19}. Such results may provide new insights into sub-Riemannian geometry.
\appendix
\section{Proof of Theorem~\ref{thm:lagrange-barycenter}}
Let $x_1,\ldots,x_m\in M$, and let $y$ be a barycenter of $(x_1,\ldots,x_m)$. For $i=1,\ldots,m$, let $\gamma_i$ be a minimizing geodesic from $x_i$ to $y$, with control $u_i$. Set $u=(u_1,\ldots,u_m)$. Then $(u,y)$ minimizes
\[
  (v,z)\in\Bigl(\prod_{i=1}^m\mathcal U_{x_i}\Bigr)\times M
  \longmapsto \mathcal J(v,z)=\frac12\sum_{i=1}^m\int_0^1|v_i(t)|^2\,dt,
\]
subject to $E_{x_i}(v_i)=z$ for every $i$, where $\mathcal U_{x_i}$ is the domain of the endpoint map $E_{x_i}$.
Choose a coordinate chart $\psi:W\to\R^n$ around $y$ and restrict each $\mathcal U_{x_i}$ to a neighborhood of $u_i$ on which $E_{x_i}$ takes values in $W$. For $z\in W$, define
\[
  \Phi(v,z)=\bigl(\psi(E_{x_1}(v_1))-\psi(z),\ldots,\psi(E_{x_m}(v_m))-\psi(z)\bigr).
\]
Then the constraints $E_{x_i}(v_i)=z$ for every $i$ are equivalent to $\Phi(v,z)=(0,\ldots,0)$.

Applying the Lagrange multiplier rule \cite[Theorem~B.2]{Rif14}, we obtain $\xi=(\xi_1,\ldots,\xi_m)\in((\R^n)^*)^m$ and $\lambda_0\in\{0,1\}$, with $(\xi,\lambda_0)\ne(0,0)$, such that
\[
  \xi^*D_{(u,y)}\Phi=\lambda_0D_{(u,y)}\mathcal J.
\]
Set $\eta_i=(D_y\psi)^*\xi_i\in T_y^*M$. Since $D_y\psi$ is an isomorphism, $(\eta,\lambda_0)\ne(0,0)$. Evaluating the multiplier identity on $(h_1,\ldots,h_m,w)$, where $h_i\in L^2([0,1];\R^k)$ and $w\in T_yM$, yields
\[
  \sum_{i=1}^m\eta_i\bigl(D_{u_i}E_{x_i}(h_i)\bigr)-\sum_{i=1}^m\eta_i(w)
  =\lambda_0\sum_{i=1}^m\int_0^1\langle u_i(t),h_i(t)\rangle\,dt.
\]
Taking $h_i=0$ for every $i$ and using the arbitrariness of $w$, we obtain
\[
  \sum_{i=1}^m\eta_i=0.
\]
Taking $w=0$ and $h_j=0$ for all $j\ne i$ gives
\[
  \eta_i^*D_{u_i}E_{x_i}=\lambda_0\int_0^1\langle u_i,\cdot\rangle\,dt.
\]
Together with $E_{x_i}(u_i)=y$, these are the required identities.

\medskip
\noindent{\scshape \textbf{Acknowledgements.}}
The author would like to thank Bang-Xian Han for introducing him to this topic. He is also grateful to Shin-ichi Ohta and Kenshiro Tashiro for their helpful discussions and valuable advice.
This work was supported by the China Scholarship Council
under Grant No.~202506270114.

\medskip
\noindent{\scshape \textbf{Use of AI declaration.}}
The proof strategies in this paper were developed by the author.
ChatGPT (OpenAI) was used to construct the two counterexamples
in Section~4.2. These counterexamples helped the author reassess
the direction of the proofs. In particular, they led the author
to abandon the attempt to establish the balance condition without
the ideality assumption. ChatGPT also provided the proof of
Theorem~\ref{thm:lagrange-barycenter} and was used for language editing and proof checking.
The author takes full responsibility for the content of this paper.
\phantomsection


\begin{thebibliography}{99}
\bibitem{AC11}
Agueh M, Carlier G.
Barycenters in the Wasserstein space.
SIAM J. Math. Anal., 2011, 43: 904--924.

\bibitem{ABB20}
Agrachev A, Barilari D, Boscain U.
A Comprehensive Introduction to Sub-Riemannian Geometry:
From the Hamiltonian Viewpoint.
Cambridge Studies in Advanced Mathematics, vol.~181.
Cambridge University Press, Cambridge, 2020.

\bibitem{BR19}
Barilari D, Rizzi L.
Sub-Riemannian interpolation inequalities.
Invent. Math., 2019, 215: 977--1038.

\bibitem{FF10}
Fathi A, Figalli A.
Optimal transportation on non-compact manifolds.
Israel J. Math., 2010, 175: 1--59.

\bibitem{FR10}
Figalli A, Rifford L.
Mass transportation on sub-Riemannian manifolds.
Geom. Funct. Anal., 2010, 20: 124--159.

\bibitem{GS98}
Gangbo W, \'{S}wi\k{e}ch A.
Optimal maps for the multidimensional Monge--Kantorovich problem.
Comm. Pure Appl. Math., 1998, 51: 23--45.

\bibitem{HL26a}
Han B-X, Liu D-Y.
Wasserstein barycenter convexity detects Hilbertian geometry.
arXiv:2606.28213, 2026.

\bibitem{HL26b}
Han B-X, Liu D-Y.
On the geometry of Wasserstein barycenter II:
Riemannian rigidity, essential non-branching, and Finsler models.
arXiv:2609.19564, 2026.

\bibitem{HLZ24}
Han B-X, Liu D-Y, Zhu Z-N.
On the geometry of Wasserstein barycenter I.
arXiv:2412.01190, 2024.

\bibitem{Jia17}
Jiang Y.
Absolute continuity of Wasserstein barycenters over Alexandrov spaces.
Canad. J. Math., 2017, 69: 1087--1108.

\bibitem{KP15}
Kim Y-H, Pass B.
Multi-marginal optimal transport on Riemannian manifolds.
Amer. J. Math., 2015, 137: 1045--1060.

\bibitem{KP17}
Kim Y-H, Pass B.
Wasserstein barycenters over Riemannian manifolds.
Adv. Math., 2017, 307: 640--683.

\bibitem{MM}
Montanari A, Morbidelli D.
On the lack of semiconcavity of the subRiemannian distance in a class
of Carnot groups.
J. Math. Anal. Appl., 2016, 444: 1652--1674.

\bibitem{Ohta12}
Ohta S-i.
Barycenters in Alexandrov spaces of curvature bounded below.
Adv. Geom., 2012, 12: 571--587.

\bibitem{Pass14}
Pass B W.
Multi-marginal optimal transport and multi-agent matching problems: uniqueness and structure of solutions.
Discrete Contin. Dyn. Syst., 2014, 34(4): 1623--1639.


\bibitem{PPV21}
Pass B, Pinamonti A, Vedovato M.
Multi-marginal optimal transport on the Heisenberg group.
Methods Appl. Anal., 2021, 28: 61--76.

\bibitem{Rif14}
Rifford L.
Sub-Riemannian Geometry and Optimal Transport.
SpringerBriefs in Mathematics.
Springer, Cham, 2014.

\bibitem{Villani09}
Villani C.
Optimal Transport: Old and New.
Grundlehren der mathematischen Wissenschaften, vol.~338.
Springer, Berlin, 2009.

\end{thebibliography}
\end{document}